\documentclass[english]{article}
\usepackage[T1]{fontenc}
\usepackage[latin9]{inputenc}
\usepackage{color}
\usepackage{amsthm}
\usepackage{amsmath}
\usepackage{amssymb}
\usepackage{esint}
\usepackage{latexsym}
\usepackage{graphicx}
\usepackage{mathscinet}
\usepackage{ulem}
\usepackage{mathtools} \mathtoolsset{showonlyrefs} 
\usepackage{enumitem}

\usepackage{csquotes}
\MakeOuterQuote{"}
\makeatletter
\theoremstyle{plain}
	\newtheorem{thm}{\protect\theoremname}[section]
  \theoremstyle{plain}
  
  \newtheorem*{cor*}{\protect\corollaryname}
  \newtheorem{lem}[thm]{\protect\lemmaname}
  \newtheorem*{lem*}{\protect\lemmaname}
	\newtheorem{prop}[thm]{\protect\propositionname}
	\newtheorem*{prop*}{\protect\propositionname}
  \theoremstyle{remark}
  
\usepackage{hyperref}
\usepackage{nameref}
\hypersetup{pdfpagemode=UseNone}

\DeclareMathOperator{\tr}{Tr}
\DeclareMathOperator{\iid}{Id}

\DeclareMathOperator{\intr}{Int}

\makeatletter
\let\orgdescriptionlabel\descriptionlabel
\renewcommand*{\descriptionlabel}[1]{%
  \let\orglabel\label
  \let\label\@gobble
  \phantomsection
  \edef\@currentlabel{#1}%
  \let\label\orglabel
  \orgdescriptionlabel{#1}%
}
\makeatother

\makeatother

\usepackage{babel}
	\providecommand{\corollaryname}{Corollary}
  \providecommand{\lemmaname}{Lemma}
  \providecommand{\remarkname}{Remark}
	\providecommand{\theoremname}{Theorem}
\providecommand{\propositionname}{Proposition}
\begin{document}

\title{Non-transversal intersection of free and fixed boundary for fully nonlinear elliptic operators in two dimensions}
\date{}

\author{E. Indrei and A. Minne}

\def\signei{\bigskip\begin{center} {\sc Emanuel Indrei\par\vspace{3mm}Center for Nonlinear Analysis\\  
Carnegie Mellon University\\
Pittsburgh, PA 15213, USA\\
email:} {\tt eindrei@msri.org }
\end{center}}

\def\signam{\bigskip\begin{center} {\sc Andreas Minne \par\vspace{3mm}
Department of Mathematics\\
KTH, Royal Institute of Technology\\
100 44 Stockholm, Sweden\\
email:} {\tt minne@kth.se}
\end{center}}

\maketitle

\begin{abstract}
In the study of classical obstacle problems, it is well known that in many configurations the free boundary intersects the fixed boundary tangentially. The arguments involved in producing results of this type rely on the linear structure of the operator. In this paper we employ a different approach and prove tangential touch of free and fixed boundary in two dimensions for fully nonlinear elliptic operators. Along the way, several $n$-dimensional results of independent interest are obtained such as BMO-estimates, $C^{1,1}$ regularity up to the fixed boundary, and a description of the behavior of blow-up solutions.
\end{abstract}

\makeatletter
\def\blfootnote{\gdef\@thefnmark{}\@footnotetext}
\makeatother


\section{Introduction}

Optimal interior regularity results have recently been obtained for
solutions to fully nonlinear free boundary problems \cite{FS, IM} via methods inspired by \cite{MR2999297}. Under further thickness assumptions, these results imply $C^{1}$
regularity of the free boundary. However, a description of the dynamics on how the free boundaries intersect the fixed boundary has remained an open problem for at least a decade in the fully nonlinear setting (although partial results have been obtained in \cite{MR2065018} under strong density and growth assumptions involving the solutions and a homogeneity assumption on the operator). On the other hand, extensive work has been carried out to investigate this question for the classical problem
\begin{equation}  \label{eq:classical}
\begin{cases}
\Delta u=\chi_{u > 0} & \text{in }B_1\cap \{x_n>0\},\\
u=0 & \text{on }\{x_n=0\},\
\end{cases}
\end{equation}
and its variations \cite{MR1359745,MR1950478,MR2180300,MR2237208,MR2281197}. The conclusions have varied as a function of the boundary data, but in the homogeneous case it has been shown that the free boundary touches the fixed boundary tangentially. Dynamics of this type have also been the object of study in the classical dam problem \cite{MR597551,MR693780} which is a mathematical model describing the filtration of water through a porous medium split into a wet and dry part via a free boundary.     

The methods utilized in establishing the above-mentioned results strongly rely on the linear structure of the operator, e.g. in arguments involving Green's functions and monotonicity formulas. In particular, the Alt-Caffarelli-Friedman and Weiss monotonicity formulas are frequently applied: tools only available in the setting of linear operators in divergence form, see \cite[Chapter 8]{MR2962060}. Thus the tangential touch problem for fully nonlinear operators requires a different approach.  

In this article we prove non-transversal intersection of free and fixed boundary in two dimensions for a broad class of fully nonlinear elliptic free boundary problems.  
More precisely, consider the following problem
\begin{equation}  \label{me}
\begin{cases}
F(D^{2}u)=1 & \text{a.e. in }B_{1}^{+}\cap\Omega,\\
|D^{2}u|\le K & \text{a.e. in }B_{1}^{+}\backslash\Omega,\\
u=0 & \text{on }B'_{1},
\end{cases}
\end{equation}
where $\Omega \subset B_1^+$ is open, $K>0$, $F$ is $C^1$, and satisfies standard structural assumptions (see \S \ref{sec:C11}). We assume solutions $u$ to be in $W^{2,p}(B_{1}^+)$ for any $1< p <\infty$. A heuristic description of our strategy is as follows: we consider $$M:= \limsup_{|x| \rightarrow 0} \frac{1}{x_n} \displaystyle \sup_{e \in \mathbb{S}^{n-2} \cap e_n^{\perp}} \partial_e u(x).$$ By extending interior $C^{1,1}$ results (see \S \ref{sec:C11}), it follows that $M$ is finite and we extract information on the nature of blow-up solutions by considering possible values for $M$. In particular, we show that either all blow-ups are of the form $bx_n^2$ if $M=0$, or there is a sequence producing a blow-up having the form $ax_1x_n+bx_n^2$ if $M \neq 0$ (Theorem \ref{alt}). 

We then show that in $\mathbb{R}_+^2$, if $ax_1x_n+bx_n^2$ is a blow-up solution, then $\partial (\intr\{u=0\})$ stays away from the origin (Lemma \ref{lem:uniquenessofblowups2D}) and this enables us to prove that blow-ups at the origin are unique (Theorem \ref{lem:uniquenessofblowups2Dexplicit}). Thereafter, a standard argument readily yields non-transversal intersection of the free and fixed boundary at contact points (Theorem \ref{tt}).

The rest of the paper is organized as follows: in \S \ref{not} we set up the problem and discuss relevant notation; \S \ref{sec:main} is the core of the paper where we rigorously develop the heuristics described above; \S \ref{sec:C11} is devoted to the $C^{1,1}$ regularity up to the boundary of solutions, which follows as in \cite{IM} once a suitable BMO result is established. The results of \S \ref{sec:C11} are used in \S \ref{sec:main}. We have chosen to reverse the logical ordering of these sections in order to make the tangential touch section more accessible.

\paragraph{Acknowledgements}  We thank Henrik Shahgholian for bringing this problem to our attention. Moreover, we thank John Andersson for his help: the technique developed in this paper evolved through our interaction with him. We also thank the Magnuson Foundation for supporting this work. E. Indrei acknowledges partial support from NSF Grants OISE-0967140 (PIRE), DMS-0405343, and DMS-0635983 administered by the Center for Nonlinear Analysis at Carnegie Mellon University. Lastly, the excellent research environment provided by KTH, CMU, HIM, and Universit\"at Bonn is kindly acknowldeged. 

\subsection{Setup and Notation} \label{not}
We study fully nonlinear elliptic partial differential equations of the form
\begin{equation}
\begin{cases}
F(D^{2}u,x)=f(x) & \text{a.e. in }B_{1}^{+}\cap\Omega,\\
|D^{2}u|\le K & \text{a.e. in }B_{1}^{+}\backslash\Omega,\\
u=0 & \text{on }B'_{1},
\end{cases}\label{eq:main}
\end{equation}

\noindent where $\Omega\subseteq B_{1}^{+}$ is an open set, $B_{1}(x)=\{x\in\mathbb{R}^{n}:|x|<1\}$,
$B_{r}^{+}(x)=B_{r}(x)\cap\{x_{n}>0\}$, $B'_{r}(x)=B_{r}(x)\cap\{x_{n}=0\}$, and $B_{r}=B_{r}(0)$. Furthermore, $F$ is assumed to satisfy the following structural conditions.
\begin{description}
\item [{(H1)\label{hyp:0@0}}] $F(0,x)\equiv0$.
\item [{(H2)\label{hyp:uniformellipticity}}]$F$ is uniformly
elliptic with ellipticity constants $\lambda_{0}$, $\lambda_{1}>0$
such that
\[
\mathcal{P}^{-}(M-N)\le F(M,x)-F(N,x)\le\mathcal{P}^{+}(M-N),\qquad\forall x\in B_{1}^{+}
\]
where $M$ and $N$ are symmetric matrices and $\mathcal{P}^{\pm}$
are the Pucci operators
\[
\mathcal{P}^{-}(M):=\inf_{\lambda_{0}\iid\le N\le\lambda_{1}\iid}\tr NM,\qquad\mathcal{P}^{+}(M):=\sup_{\lambda_{0}\iid\le N\le\lambda_{1}\iid}\tr NM.
\]

\item [{(H3)\label{hyp:concavity}}] $F(\cdot, x)$ is concave or convex for all $x\in B_{1}^{+}$.
\item [{(H4)\label{hyp:Freg}}] 
\begin{equation}
|F(M,x)-F(M,y)|\le\overline{C}(|M|+1)|x-y|^{\bar{\alpha}}\label{eq:H4}
\end{equation}
for some $\bar{\alpha}\in(0,1]$ and $x$, $y\in B_{1}^+$.
\end{description}
Moreover, let 
\[
\beta(x,x^0):=\sup_{M\in\mathcal{S}}\frac{|F(M,x)-F(M,x^0)|}{|M|+1}
\]
where $\mathcal{S}$ is the space of $n\times n$ symmetric real valued
matrices. 

\paragraph{Notation} Points in $\mathbb{R}^n$ are generally denoted by $x$, $x^0$, $y$ etc. while subscripts are used for components, i.e. $x=(x_1,\ldots,x_n)$, scalar sequences, and functions. The notation $x'$ is used for $(n-1)$-dimensional vectors. Similarly,
$\nabla$ and $\nabla'$ will be used, respectively, for the gradient and the gradient with respect to the first $n-1$ variables.
\begin{align*}
  \mathbb{R}_+^n&\quad \text{is the upper half space } \{x\in \mathbb{R}^n\,:\,x_n>0\};\\
	\Omega &\quad \text{is an open set in }\mathbb{R}_+^n;\\
	\Gamma&\quad \text{is the set }\mathbb{R}_+^n\cap\partial \Omega;\\
	\Gamma_i&\quad \text{is the set }\mathbb{R}_+^n\cap\partial \intr \{u=0\};\\
	B_r(x^0) &\quad \text{is the open ball }\{x\in \mathbb{R}^n\,:\,|x-x^0|<r\};\\
	B_r^+(x^0) &\quad \text{is the truncated open ball }\{x\in \mathbb{R}^n\,:\,|x-x^0|<r,\,x_n>0\};\\
	B_r'&\quad \text{is the ball }\{x'\in \mathbb{R}^{n-1}\,:\,|x'|<r\};\\
	\mathbb{S}^{n-1}&\quad \text{is the $(n-1)$-sphere }\{x\in \mathbb{R}^{n}\,:\,|x|=1\};\\
	e^\perp&\quad \text{is the vector space orthogonal to $e\in \mathbb{S}^{n-1}$};\\
	C^{k,\alpha}(\Omega)&\quad \text{denotes the usual Hölder space;}\\
	C_{\text{loc}}^{k,\alpha}(\Omega)&\quad \text{denotes the local Hölder space;}\\
	W^{k,p}(\Omega)&\quad \text{denotes the usual Sobolev space.}
\end{align*}
The term "blow-ups of $u$ at $x^0$" will be used for limits of the form $\displaystyle \lim_{j\to \infty}\frac{u(x^0+r_j x)}{r_j^2}$, where $r_j$ is a sequence such that $r_j \rightarrow 0^+$ as $j \rightarrow \infty$; $\intr \{u=0\}=\{u=0\}^\circ$ means the interior of the set $\{u=0\}:=\{x\in \mathbb{R}_+^n\,:\,u(x)=0\}$. Finally, $S(\varphi)$ denotes the space of viscosity solutions corresponding to $\varphi$ and the ellipticity constants $\lambda_0$ and $\lambda_1$ in \hyperref[hyp:uniformellipticity]{(H2)}, see \cite{MR1351007}.

\section{Main Result}\label{sec:main}

Our first result gives a natural dichotomy of blow-ups of solutions to \eqref{me} in any dimension.
\begin{thm}[Blow-up Alternative] \label{alt}
Let $u$ be a solution to \eqref{me} and suppose $\{\nabla u \neq 0\} \cap \{x_n>0\} \subset \Omega$, $0 \in \overline{\{u \neq 0\}}$, and $\nabla u(0)=0$. Then exactly one of the following holds: 
\begin{description}[leftmargin=0pt]
\item[(i)\label{alt(i)}]  All blow-ups of $u$ at the origin are of the form $u_0(x)=b x_n^2$ for some $b >0$;\\
\item[(ii)\label{alt(ii)}] There exists a blow-up of $u$ at the origin of the form
\[
u_0(x)=ax_1x_n+ b x_n^2,
\]
for $a \neq 0$, $b \in \mathbb{R}$.
\end{description}
\end{thm}

\begin{proof} 
Firstly, since $u(x',0)=0$, it follows that $\partial_{x_i} u (x',0) = 0$ for all $i\in \{1,\ldots,n-1\}$. By $C^{1,1}$ regularity (Theorem \ref{thm:C11}), there is a constant $C>0$ such that $$\bigg | \frac{1}{x_n} \partial_{x_i} u(x',x_n)\bigg| =\bigg | \frac{1}{x_n} \big( \partial_{x_i} u(x',x_n) - \partial_{x_i} u(x',0)\big) \bigg| \le C,\qquad x_n>0.$$ Define
\[ 
M:= \limsup_{\begin{subarray}{c} |x| \to 0 \\ x_n>0\end{subarray}} \frac{1}{x_n} \displaystyle \sup_{e \in \mathbb{S}^{n-2} \cap e_n^{\perp}} \partial_e u(x).
\]
In particular, $0\le M \le C<\infty$ and there exists a sequence $x^j \rightarrow 0$ with $x_n^j>0$ and directions $e_{x^j} \in \mathbb{S}^{n-2}$ such that $$\lim_{j \rightarrow \infty} \frac{1}{x_n^j} \displaystyle \partial_{e_{x^j}} u(x^j)=M.$$ Moreover, there exists $e \in \mathbb{S}^{n-2}$ such that  (up to a subsequence) $e_{x^j} \rightarrow e.$ Next note 
\begin{align*}
\bigg |\frac{1}{x_n^j} \nabla' u(x^j) \cdot e - M \bigg| &\le \bigg |\frac{1}{x_n^j} \nabla' u(x^j) \cdot (e-e_{x^j})\bigg |+\bigg |\frac{1}{x_n^j} \nabla' u(x^j) \cdot e_{x^j} - M \bigg|\\
& \le C|e-e_{x^j}| + \bigg |\frac{1}{x_n^j} \nabla' u(x^j) \cdot e_{x^j} - M\bigg| \rightarrow 0,
\end{align*}
as $j \rightarrow \infty$. Thus, up to a rotation, 
$$\lim_{j \rightarrow \infty} \frac{1}{x_n^j} \displaystyle \partial_{x_1} u(x^j)=M.$$ 
Now consider a sequence $\{s_j\}$ such that $s_j \rightarrow 0^+$ and the corresponding blow-up procedure so that $$u_j(x):=\frac{u(s_j x)}{s_j^2} \rightarrow u_0(x)$$ in $C_{\text{loc}}^{1,\alpha}(\mathbb{R}_+^n)$ for any $\alpha \in [0,1)$, and $u_0$ satisfies 
\begin{equation} \label{m2}
\begin{cases}
F(D^{2}u_0)=1 & \text{a.e. in }\mathbb{R}_+^n\cap\Omega_0,\\
|\nabla u_0|=0 & \text{in }\mathbb{R}_+^n\backslash\Omega_0,\\
u=0 & \text{on }\mathbb{R}_+^{n-1},
\end{cases}
\end{equation}
where $\Omega_0 =\{\nabla u_0 \neq 0 \} \cap \{x_n>0\}$ (via non-degeneracy). The definition of $M$ implies 
\begin{equation} \label{part}
M\ge \lim_j \bigg| \frac{\partial_{x_i}  u(s_j x)}{s_j x_n}\bigg| = \lim_j \bigg| \frac{\partial_{x_i} u_j(x)}{x_n} \bigg| =\bigg|\frac{\partial_{x_i} u_0(x)}{x_n}\bigg|,
\end{equation}
for all $i \in \{1,\ldots, n-1\}$. In particular, let $v= \partial_{x_1} u_0$ so that in $\mathbb{R}_+^n$, 
\begin{equation} \label{ine}
|v(x)| \le Mx_n.
\end{equation}
If $M=0$, then \eqref{part} implies $\partial_{x_i} u_0 =0$ for all $i \in \{1,\ldots, n-1\}$ so that $u_0(x)=u_0(x_n)$. However, since $u_0(0)=|\nabla u_0(0)|=0$, $0\in \overline{\{u_0\ne 0\}}$ and $u_0$ satisfies \eqref{m2}, the uniform ellipticity of $F$ readily implies
\[
u_0(x)= bx_n^2,
\]
for some $b>0$. This shows that if $M=0$, then any blow-up at the origin is of the form stated in (i). 

Now suppose $M>0$. In order to prove (ii), we cook up a specific blow-up: let $r_j:= |x^j|$ (recall that $\{x^j\}$ is the sequence achieving the $\lim \sup$ in the definition of $M$) so  that as before $$u_j(x):=\frac{u(r_j x)}{r_j^2} \rightarrow u_0(x)$$ in $C_{\text{loc}}^{1,\alpha}(\mathbb{R}_+^n)$ for any $\alpha \in [0,1)$, and $u_0$ satisfies \eqref{m2}, \eqref{part}, and \eqref{ine}. Set $y^j = \frac{x^j}{r_j} \in \mathbb{S}^{n-1} \cap \{x_n >0\}$ and note that along a subsequence, $y^j \rightarrow y \in \mathbb{S}^{n-1}\cap \{x_n\ge 0\}$. Moreover, by the choice of the sequence $\{x^j\}$ and the $C^{1,\alpha}$ convergence of $u_j$ to $u_0$, if $y_n>0$, then 
\begin{equation} \label{lim}
\displaystyle \lim_j \frac{v(y^j)}{y_n^j}= \displaystyle \lim_j \frac{\partial_{x_1} u_j(y^j)}{y_n^j} = \displaystyle \lim_j  \frac{\partial_{x_1} u(x^j)}{x_n^j} = M,
\end{equation} 
so that 
\begin{equation} \label{eq}
v(y)=My_n,
\end{equation}
and note that \eqref{eq} also holds if $y_n=0$.
We consider several possibilities keeping in mind that $M>0$.
\begin{description}[leftmargin=0pt]
\item[Case 1:] If $y \in \Omega_0$, then by differentiating \eqref{m2} we get the elliptic equation
\[
a_{ij} \partial_{ij} (v(x)-Mx_n)=0
\]
for some measurable $a_{ij}$, and by \eqref{ine}, \eqref{eq}, and the maximum principle, it follows that $v(x)=Mx_n$ in the connected component of $\Omega_0$ containing $y$, say $\Omega_0^y$. If there exists $x \in \partial \Omega_0^y \cap \{x_n>0\}$, then $Mx_n=v(x)=0$ so we must have $M=0$, a contradiction. Thus, $v(x)=Mx_n$ in $\mathbb{R}_+^n$ and by integrating,
\[
u_0(x)=Mx_1x_n+h(x_2,\ldots, x_n).
\]
Now, Krylov/Safonov's up to the boundary $C^{2,\alpha}$ estimate (see e.g. Theorem \ref{lem:C2alpha}) applied to $u_0(Rx)/R^2$ yields 
\[
\frac{|D^2u_0(x)-D^2u_0(y)|}{|x-y|^\alpha}\le \frac{C}{R^{\alpha}},\qquad x \neq y\in B_R^+,
\]
and taking $R\to \infty$ implies that $D^2 u_0$ is a constant matrix and thus $h$ is a second order polynomial.
Since $u_0$ vanishes on $\{x_n=0\}$, it follows that 
$$h(x_2,\ldots, x_n)= x_n \sum_{i\neq n} \alpha_i x_i + b x_n^2,$$ and so up to a rotation, $$u_0(x)=a x_1x_n+b x_n^2,$$ with $a$ or  $b\ne 0$.  
\item[Case 2:] If $y \in \partial \Omega_0 \cap \{x_n>0 \}$, then $My_n=v(y)=0$, a contradiction.
\item[Case 3:] If $y \in \overline{\Omega}_0^c$, then for all but finitely many $j$, $y^j \in \Omega_0^c$ and since $\{\nabla u_0 \neq 0 \} \subset \Omega_0$, it follows that $v(y^j)=0$ if $j \ge N$ for some $N \in \mathbb{N}$. However, $y_n^j>0$ and so $0 = \displaystyle \lim_j \frac{v(y^j)}{y_n^j}=M$, a contradiction. 
\item[Case 4:] If $y \in \partial \Omega_0 \cap \{x_n=0\}$, by differentiating \eqref{m2} in $\Omega_0$, it can be seen that for $r>0$ (to be picked later), $v$ satisfies 
\[
Lv=0 \,\,\, \text{in } B_r(y)^+ \cap \Omega_0,\\
\]
where $L=F_{ij}(D^2 u_0) \partial_{ij}$ is elliptic. Since $u_0 \in C^{1,1}(B_r^+(y))$, it follows that the $F_{ij}(D^2 u_0)$ are bounded and measurable on $B_r^+(y)$.

\noindent We know that $Mx_n -v(x)\ge 0$  in $\mathbb{R}^n_+$, and if equality holds everywhere, $u_0(x)=ax_1 x_n + bx_n^2$ just as in Case 1. If there is a point $z$ where strict inequality holds, $Mz_n -v(z)> 0$, we can choose a ball $B_r^+(y)$ so that, by continuity of $v$, $v(x)<Mx_n$ in a neighborhood $B_s(z)$, where $z$ is a boundary point of $B_r^+(y)$. Note that this strict inequality necessarily occurs on $\partial B_r^+(y)\cap \{x_n>0\}$ since both $v$ and $Mx_n$ are zero on the hyperplane $\{x_n=0\}$. Now choose a smooth non-negative (but not identically zero) function $\phi$ supported on $B_s(z)$ small enough such that $Mx_n-\phi(x)\ge v(x)$ in $\mathbb{R}^n_+$ and $Mx_n-\phi(x)>0$ (this can be done since $B_s(z)$ is some distance away from the hyperplane $\{x_n=0\}$). Then if
\begin{equation} \label{m3}
\begin{cases}
Lw=0 & \text{in } B_r^+(y),\\
w=Mx_n-\phi & \text{on } \partial B_r^+(y),\\
\end{cases}
\end{equation}
we have that $w>0$ in $B_r^+(y)$ by the strong maximum principle since $Mx_n-\phi(x)>0$. In particular, $w>v=0$ on $\partial\Omega$, and since $v\le w$ on $\partial B_r^+(y)$, the strong maximum principle again gives $w>v$ in $B_r^+(y)\cap \Omega$. Note also by linearity that $w=Mx_n -h$ where $h$ solves
\begin{equation} \label{m4}
\begin{cases}
Lh=0 & \text{in } B_r^+(y),\\
h=\phi & \text{on } \partial B_r^+(y),\\
\end{cases}
\end{equation}
Once more, the strong maximum principle shows that $h>0$ in $B_r^+(y)$, so the boundary Harnack comparison principle implies that $cx_n\le h(x)$ in $B_{r/2}^+(y)$, where $c>0$ depends on ellipticity constants and $\phi$. With this,
\[
	M=\lim_{j\to \infty} \frac{v(y^j)}{y^j_n}\le \limsup_{\begin{subarray}{c}x_n\to 0^+ \\ x\in B_{r/4}^+(y)\end{subarray}} \frac{w(x)}{x_n}\le \lim_{\begin{subarray}{c}x_n\to 0^+ \\ x\in B_{r/4}^+(y)\end{subarray}} \frac{Mx_n-cx_n}{x_n} = M-c,
\]
a contradiction.
\end{description}
\end{proof}

The next lemma shows that in two dimensions, if \hyperref[alt(ii)]{(ii)} in Theorem \ref{alt} occurs, then $\Gamma_i=\mathbb{R}_+^n\cap\partial \intr \{u=0\}$ stays away from the origin.

\begin{lem}
\label{lem:uniquenessofblowups2D}Let $u$ be a solution to \eqref{me}
with $\Omega = \big(\{u \ne 0\} \cup \{\nabla u \neq 0\} \big) \cap \{x_2>0\} \subset \mathbb{R}_+^2$. If there exists $\{r_j\} \subset \mathbb{R}^+$ such that $r_j \to 0$ as $j \to \infty$ and $$u_{j}(x):=\frac{u(r_jx)}{r_j^2} \to u_{0}(x)=ax_{1}x_2+bx_{2}^{2}$$ in $C_{\text{loc}}^{1,\alpha}(\mathbb{R}_+^n)$ as $j \rightarrow \infty$,
for $a \neq 0$, $b \in \mathbb{R}$, then there exists $\delta \in (0,1)$ such that
$B_{\delta}^{+}\cap \Gamma_i =\emptyset$.
\end{lem}

\begin{proof}
We may assume $a>0$. Set $v_j:= \partial_1 u_j$ and let $R>2$, $\mu \in (0,\frac{1}{4})$, and $\delta \in (0,\frac{1}{4})$. Then select $j_0=j_0(R, \mu, \delta)>0$ such that for all $j \ge j_0$ 

\begin{equation} \label{c1}
|\nabla u_j(x)| >0, \hskip .1in x \in B_R^+ \setminus B_\delta^+;
\end{equation}

\begin{equation} \label{c2}
v_j(x) > 0, \hskip .1in x \in B_R^+ \cap \{x_2 \ge \mu\}.
\end{equation}
(the two-dimensional setting is crucial for \eqref{c1}). Consider $z \in \partial B_1 \cap \{x_2=0\}$ and note that $$B_{\frac{3}{4}}^+(z) \subset B_R^+ \setminus B_\delta^+.$$Thanks to \eqref{c1}, $u_j$ satisfies $F(D^2 u_j)=1$ in $B_{\frac{3}{4}}(z)^+$ for all $j \ge j_0$. $C^{2,\alpha}$ estimates up to the boundary (see Theorem \ref{lem:C2alpha}) implies 
\[
\sup_j \|u_j\|_{C^{2,\alpha}\big(B_{\frac{3}{4}}^+(z)\big)} < \infty.
\] 
Thus, along a subsequence, $v_j \rightarrow ax_2$ in $C^{0,1}$ ($C^{2,\alpha}$ is compactly contained in $C^{1,1}$) and so   

\[
c_j:= \sup_{\begin{subarray}{c}x,y\in B_{3/4}^+(z)\\ x\ne y\end{subarray}} \frac{|(v_j(x)-v_j(y))-(v(x)-v(y))|}{|x-y|} \rightarrow 0.
\]
In particular, since $v_j(x_1,0)=v(x_1,0)=0$, it follows that $$\frac{|v_j(x)-ax_2|}{x_2} \le c_j$$ and so 

\[
v_j(x) \ge (a-c_j)x_2.
\]
Now we select $j$ large such that $v_j(x) \ge 0$ on $\partial B_1$.    
Note that $Lv_j=0$ in $B_{1}^{+}\cap\Omega(u_{j})$ where $L$ is an elliptic second order operator obtained by differentiating \eqref{me}.
Indeed, $u_j$ satisfies 
\begin{equation}
\begin{cases}
F(D^{2}u_j)=1 & \text{a.e. in }B_{1/r_j}^{+}\cap\Omega(u_j),\\
|D^{2}u|\le K & \text{a.e. in }B_{1/r_j}^{+}\backslash\Omega(u_j),\\
u_j=0 & \text{on }B'_{1/r_j},
\end{cases}\label{eq:main2}
\end{equation}
where $\Omega(u_j)$ is the dilated set $\Omega/r_j$, and without loss we may assume $r_j<\frac{1}{2}$. Since $v_j$ vanishes on $\partial \Omega(u_{j})$ and is non-negative on $\partial B_1^+$, the maximum principle implies 
$v_j>0$ in $B_{1}^{+} \cap \Omega(u_{j})$ (note that $v_j$ is not identically zero by \eqref{c2}). If $\Gamma_i(u_j) \cap B_\delta^+ \neq \emptyset$, consider a ball $N$ in the interior of $\{u_j=0\} \cap B_{\delta}^{+}$. For $t \in \mathbb{R}$, let $N_{t} = N + te_1$. Note that by taking $t$ negative we can move $N_t$ to the left so that eventually $N_{t} \subset B_{1}^{+}\backslash B_{\delta}^{+}$. Consider the strip $S=\cup_{t\in \mathbb{R}} N_t$. The next claim is that  there exists a ball in the set $(S \cap B_1^+) \setminus B_\delta^+$ such that $u_j \neq 0$ in this ball: if not, then for each point $z \in (S \cap B_1^+) \setminus B_\delta^+$ there exists a sequence $\{z_k\} \subset \{u_j=0\}$ such that $z_k \rightarrow z$; by continuity, $u_j(z)=0$, so $u_j=0$ in $(S \cap B_1^+) \setminus B_\delta^+$ and therefore the gradient also vanishes there, a contradiction to \eqref{c1}. Denote the ball by $N_{\tilde t} \subset \Omega(u_j)$ and note that $u_j <0$ on $N_{\tilde t}$ since for each $z \in N_{\tilde t}$, there exists $t_z>0$ such that $z+e_1t_z \in \{u_j=0\}$ and $v_j > 0$ in $B_1^+ \cap \Omega(u_j)$. Thus, $N_{\tilde t} \subset \Omega(u_j) \cap \{u_j<0\}$. Now move $N_{\tilde t}$ to the right until the first time it touches $\{u_j=0\}$, and let $y$ be the contact point. If $\nabla u_j(y) =0$, we immediately obtain a contradiction via Hopf's lemma. Thus we may assume $\nabla u_j(y) \neq 0$ which implies $y \in \Omega(u_j)$; whence $v_j(y)>0$ (recall that $v_j>0$ in $\Omega(u_{j})$). By continuity $v_j>0$ in $B_r(y)$ for some $r>0$ so in particular $v_j(y+te_1)>0$ for all $t>0$ small. Since $\{y+te_1: t \in (0, r)\} \subset \Omega(u_j)$, $t_*:=\sup \{t>0: y+te_1 \in \Omega(u_{j})\}$ is positive. Note that $y+te_1$ will eventually enter $N$ as $t$ gets larger. However, $$u_j(y+t_*e_1)- u_j(y)=\int_{0}^{t_*} v_j(y+s e_1) ds >0,$$ and this implies $0=u_j(y+t_*e_1)> u_j(y)=0$, a contradiction. Thus $\Gamma_i(u_j) \cap B_\delta^+ = \emptyset$ and the result follows.              
\end{proof}

\noindent Before proving uniqueness of blowups and tangential touch, we require one more lemma. 

\begin{lem}\label{lem:intempty}
Let $u$ be a solution to \eqref{me} with $\Omega = \big(\{u \ne 0\} \cup \{\nabla u \neq 0\} \big) \cap \{x_n>0\}$. If $s \in (0,1]$ and $(B_s^+ \setminus \Omega)^\circ = \emptyset$, then $|B_s^+ \setminus \Omega|=0$.  
\end{lem}

\begin{proof}
Since $u \in W^{2,n}(B_1^+)$, it follows that $D^2 u = 0$ a.e. on $B_s^+ \setminus \Omega$. Let $Z:=\{D^2 u =0\} \cap (B_s^+ \setminus \Omega)$ and note that $|Z|=|B_s^+ \setminus \Omega|$. Thus if $Z \subset (B_s^+ \setminus \Omega)^\circ$, then the result follows. Let $x^0 \in Z$ and suppose $x^0 \notin (B_s^+ \setminus \Omega)^\circ$. Then consider a sequence of points $x^j \rightarrow x^0$ such that $u(x^j) \neq 0$ and let $r_j:=|x^0-x^j|$. Non-degeneracy (see e.g. Lemma 3.1 in \cite{IM}) implies that for $j$ large, $$\displaystyle \sup_{\partial B_{r_j}(x^0)} \frac{u}{r_j^2} \ge c > 0,$$ or in other words $$\displaystyle \sup_{\partial B_{1}(0)} \frac{u(x^0+r_jx)}{r_j^2} \ge c>0.$$ Now for each $j$ large enough, let $y^j \in \partial B_1(0)$ be the element achieving the supremum in the previous expression; note that since $u(x^0)=|\nabla u(x^0)|=|D^2u(x^0)|=0$, we have $$u(x^0+r_jy^j) = o(r_j^2),$$ a contradiction.        
\end{proof}

Theorem \ref{alt}, Lemma \ref{lem:uniquenessofblowups2D}, and Lemma \ref{lem:intempty} imply uniqueness of blow-ups in two dimensions.

\begin{thm}[Uniqueness of Blow-ups]\label{lem:uniquenessofblowups2Dexplicit}
Let $u$ be a solution to \eqref{me}
with $\Omega = \big(\{u \ne 0\} \cup \{\nabla u \neq 0\} \big) \cap \{x_2>0\}\subset \mathbb{R}_+^2$. If $0 \in \overline{\{u \neq 0\}}$ and $\nabla u(0)=0$, then all blow-up limits $u_0$ of $u$ at the origin are of the form 
\[
u_0(x)=ax_1x_2+bx_2^2
\]
where $a, b\in \mathbb{R}$ with at least one of them non-zero. 
\end{thm}
\begin{proof}
We divide the proof into two cases.
\begin{description}[leftmargin=0pt]
\item[Case 1, $\bf{0\in \overline{\Gamma}_i}$:]
	Lemma \ref{lem:uniquenessofblowups2D} implies the non-existence of a blow-up $u_0$ of $u$ of the form
	\[
	ax_1x_2 + b x_2^2,
	\]
	$a \neq 0$, $b \in \mathbb{R}$ from which it follows that \hyperref[alt(i)]{(i)} holds in Theorem \ref{alt}.
\item[Case 2, $\bf{0\not\in \overline{\Gamma}_i}$:]
	In this case there exists $\delta>0$ such that $\Gamma_i \cap B_\delta^+ = \emptyset.$ Since $0 \in \overline{\{u \neq 0\}}$ (by assumption), it follows that $B_\delta^+ \not\subset \{u=0\}^\circ$ and as $\Gamma_i \cap B_\delta^+ = \emptyset,$ we may conclude that $\{u=0\}^\circ \cap B_\delta^+ = \emptyset$. Thus the hypotheses of Lemma \ref{lem:intempty} are satisfied and by applying the lemma we obtain that $F(D^2 u)=1$ a.e. in $B_\delta^+$.
	Therefore $u \in C^{2,\alpha}(B_{\delta/2}^+)$ and the blow-up limit $u_0$ is uniquely given by
	\begin{align}
	\lim_{r\to 0} \frac{u(rx)}{r^2} &= \lim_{r\to 0}\frac{u(0)+\nabla u(0)\cdot rx+\langle rx,D^2u(0) rx \rangle + o(r^2)}{r^2}\\
			&= \langle x,D^2u(0) x \rangle= ax_1x_2+bx_2^2.
	\end{align}
	The last equality follows from the boundary condition. Furthermore, $u_0$ solves the same equation as $u$ so $F(D^2u_0)=F(D^2u(0))=1$
	and so $a$ and $b$ cannot both be zero due to \hyperref[hyp:0@0]{(H1)}.
\end{description}
\end{proof}

If blow-ups are unique and of the form given above, it is rather standard to show that the free boundary touches the fixed boundary tangentially (see e.g. Chapter 8 in \cite{MR2962060}). The proof is included for completeness.

\begin{thm}[Tangential Touch] \label{tt}
Let $u$ be a solution to \eqref{me} with $\Omega = \big(\{u \ne 0\} \cup \{\nabla u \neq 0\} \big) \cap \{x_2>0\} \subset \mathbb{R}_+^2$. Then there exists a constant $r_0>0$ and a modulus of continuity $\omega_u(r)$ such that 
\[
\Gamma(u) \cap B_{r_0}^+ \subset \{x: x_2 \le \omega_u(|x|)|x|\},
\]
if $0\in \overline{\Gamma(u)}$, where $\Gamma(u) :=\partial \Omega \cap \mathbb{R}_+^2$.
\end{thm}

\begin{proof}
By Theorem \ref{lem:uniquenessofblowups2Dexplicit} the blow-up of $u$ at the origin is not identically zero and given by $u_0(x)=ax_1x_2+bx_2^2$. In particular, $\Gamma(u_0)=\emptyset$. It suffices to show that for any $\epsilon>0$ there exists $\rho_\epsilon=\rho_\epsilon(u)>0$ such that 
\[
\Gamma(u) \cap B_{\rho_\epsilon}^+ \subset B_{\rho_\epsilon}^+ \setminus \mathcal{C}_\epsilon, 
\]
where $\mathcal{C}_\epsilon := \{x_2>\epsilon|x_1|\}$. Suppose not, then there exists a solution $u$ to \eqref{me} satisfying the hypotheses of the theorem and $\epsilon>0$ such that for all $k\in \mathbb{N}$ there exists 
\[
x^k \in \Gamma(u) \cap B_{\frac{1}{k}}^+ \cap \mathcal{C}_\epsilon.
\]
Let $r_k:=|x^k|$ and $y^k :=\frac{x^k}{r_k} \in \partial B_1 \cap \mathcal{C}_\epsilon.$ Note that along a subsequence $$y^k\rightarrow y \in \partial B_1 \cap \mathcal{C}_\epsilon.$$  Define $$u_k(x):=\frac{u(r_kx)}{r_k^2}$$ so that $u_k \rightarrow u_0$ in $C_{\text{loc}}^{1,\alpha}(\mathbb{R}_+^n)$ (along a subsequence). In particular $y \in \Gamma(u_0)$ which contradicts that $\Gamma(u_0)=\emptyset$.
\end{proof}


\section{\texorpdfstring{$C^{1,1}$}{C1,1} Regularity up to the Boundary}\label{sec:C11}
In this section we show BMO-estimates as well as $C^{1,1}$ regularity up to the fixed boundary of solutions to \eqref{eq:main}.  

%

\begin{thm}[$C^{1,1}$ regularity]
\label{thm:C11}Let $f\in C^{\alpha}(B_{1}^{+})$ be a given function
and $\Omega\subseteq B_{1}^{+}$ a domain such that $u:B_{1}^{+}\to\mathbb{R}$
is a $W^{2,n}$ solution of \eqref{eq:main}. Assume $F$ satisfies
\hyperref[hyp:0@0]{(H1)}-\hyperref[hyp:Freg]{(H4)}. Then there exists a constant $C$
depending on $\|u\|_{W^{2,n}(B_{1}^{+})}$,$\|f\|_{C^{\alpha}(B_{1}^{+})},$
and universal constants such that 
\[
|D^{2}u|\le C,\qquad\text{a.e. in }B_{1/2}^{+}
\]
\end{thm}
There are three key tools needed to prove this theorem. The first two are $C^{2,\alpha}$ and $W^{2,p}$ estimates up to the boundary for the following classical fully nonlinear problem
\begin{equation}
\begin{cases}
F(D^{2}u,x)=f(x) & \text{a.e. in }B_{1}^{+},\\
u=0 & \text{on }B'_{1},
\end{cases}\label{eq:standardproblem}
\end{equation}
and the last involves BMO-estimates.
The $C^{2,\alpha}$ and $W^{2,p}$ estimates are well-known \cite{W2, Saf, MR2486925, MR0262675}. We have been unable to find a reference for the BMO-estimates and thus provide a proof which is an adaptation of the interior case. For convenience, we record the following estimates, see e.g. \cite[Theorem 4.3]{MR2486925} and \cite[Theorem 7.1]{Saf}.

\begin{thm}
[$W^{2,p}$ Regularity]\label{lem:W2p} Let $u$ be a  $W^{2,p}$ viscosity solution
to \eqref{eq:standardproblem} and $f\in L^{p}(B_{1}^{+})$ for $n\le p\le\infty$.
If $\beta(x^0,y)\le \beta_0$ in $B_r^+(x^0)\cap B_1^+$ for all $x^0\in B_1^+$ and $0<r\le r_0$, where $\beta_0$ and 
$r_0$ are universal constants, then $u\in W^{2,p}(B_{1/2}^{+})$ and
\[
\|u\|_{W^{2,p}(B_{1/2}^{+})}\le C(\|u\|_{L^{\infty}(B_{1}^{+})}+\|f\|_{L^{p}(B_{1}^{+})}),
\]
where $C=C(n,\lambda_{0},\lambda_{1},\bar{\alpha},\overline{C},p)>0$.
\end{thm}

\begin{thm}
[$C^{2,\alpha}$ Regularity]\label{lem:C2alpha} Let $u$ be a  $W^{2,n}$ viscosity solution to \eqref{eq:standardproblem}. Then if $\beta(x^0,y)\le \beta_0$ in $B_r^+(x^0)\cap B_1^+$ for all $x^0\in B_1$ and $0<r\le r_0$, where $\beta_0$ and $r_0$ are universal constants, then $u\in C^{2,\alpha}(B_{1/2}^{+})$
and
\[
\|u\|_{C^{2,\alpha}(B_{1/2}^{+})}\le C(\|u\|_{L^{\infty}(B_{1}^{+})}+\|f\|_{C^{\bar{\alpha}}(B_{1}^{+})}),
\]
where $C=C(n,\lambda_{0},\lambda_{1},\bar{\alpha},\overline{C})>0$.
\end{thm}

The next results are technical tools utilized in the proof of the BMO-estimate (i.e. Proposition \ref{lem:BMOestimate}). The first is an approximation lemma, see e.g. \cite[Lemma 1.4]{W2}.      

\begin{lem}
[Approximation]\label{lem:compactness} Let $\epsilon>0$, $u\in W^{2,p}(B_{1}^{+}(x^0))$,
and let $v$ solve
\[
\begin{cases}
F(D^{2}v,x^0)=a & \text{in }B_{1/2}^{+}(x^0),\\
v=u & \text{on }\partial B_{1/2}^{+}(x^0).
\end{cases}
\]
Then there exists $\delta>0$
and $\eta>0$ such that if
$$\beta(x,x^0):=\sup_{M}\frac{|F(M,x)-F(M,x^0)|}{|M|+1}\le\delta$$
and $|f(x)-a|\le\eta$ a.e. for $f(x):=F(D^2u(x),x)$ in $B_{1}^{+}(x^0)$, then
\[
|u-v|\le\epsilon\qquad\text{in }B_{1/2}^{+}.
\]
\end{lem}
\begin{lem}
\label{lem:rhoapprox}Let $u$ be a $W^{2,n}(B_1^+)$ solution to \eqref{eq:main}
such that $|u|\le1$, $\beta(x,y)$ satisfies \hyperref[hyp:Freg]{(H4)} and $|F(D^2u(x),x)|\le \delta$ a.e. in $B_{1}^{+}$ for $\delta$ as in Lemma 
\ref{lem:compactness}.
Then there exists a universal constant $\rho>0$ such that
\[
|D^2P_{k,x^0}-D^2P_{k-1,x^0}|\le C_0(n,\lambda_0,\lambda_1)
\]
and
\[
|u(x)-P_{k,x^0}(x)|\le \rho^{2k},\qquad\text{inside }B_{\min(\rho^{k},1)}^{+}(x^0),\, k\in\mathbb{N}_0,
\]
where $P_{k,x^0}$ is a second order polynomial such that $F(D^{2}P_{k,x^0},x^0)=0$
and $x^0\in B_{1/2}^{+}$. \end{lem}
\begin{proof}
For $k=0$ and $k=-1$, the statement is true for $P_{k,x^0}(x)\equiv0$ by
assumption (recall \hyperref[hyp:0@0]{(H1)}). If we assume it is true up to some $k$, define $u_{k}:=\frac{u(\rho^{k}x+x^0)-P_{k,x^0}(\rho^{k}x+x^0)}{\rho^{2k}}$ and
\[
F_{k}(M,x):=F(M+D^{2}P_{k,x^0},\rho^k x+x^0),\qquad x\in B_1\cap \{x_n>-\frac{x_n^0}{\rho^k}\}.
\]
Then $|F_{k}(D^{2}u_{k},x)|=|F((D^{2}u)(\rho^k x+x^0),\rho^k x+x^0)|\le\delta$ a.e. Also,
\begin{align*}
\beta_{k}(x,0) & =\sup_{M\in\mathcal{S}}\frac{|F_{k}(M,x)-F_{k}(M,0)|}{|M|+1}\\
 & =\sup_{M\in\mathcal{S}}\frac{|F(M+D^{2}P_{k,x^0},\rho^k x+x^0)-F(M+D^{2}P_{k,x^0},x^0)|}{|M|+1}\\
 & =\sup_{M\in\mathcal{S}}\frac{|F(M,\rho^k x+x^0)-F(M,x^0)|}{|M-D^{2}P_{k,x^0}|+1}\\
 & =\sup_{M\in\mathcal{S}}\frac{|F(M,\rho^k x+x^0)-F(M,x^0)|}{|M|+1}\frac{|M|+1}{|M-D^{2}P_{k,x^0}|+1}\\
 & \le\beta(\rho^k x+x^0,x^0)\sup_{M\in\mathcal{S}}\frac{|M|+1}{|M-D^{2}P_{k,x^0}|+1}\\
 & \le\overline{C}\rho^{\overline{\alpha}k}\sup_{M\in\mathcal{S}}\frac{|M|+1}{||M|-|D^{2}P_{k,x^0}||+1}\\
 & \le\overline{C}\rho^{\overline{\alpha}k}(|D^{2}P_{k,x^0}|+1),
\end{align*}
where the last inequality follows from a calculation of the maximum of the function $\frac{x+1}{|x-a|+1}$, $x$, $a\ge 0$.
However, from the induction hypothesis,
\[
|D^{2}P_{k,x^0}|\le \sum_{j=1}^k |D^{2}P_{j-1,x^0}-D^{2}P_{j,x^0}|\le C_0 k
\]
so
\[
	\overline{C}\rho^{\overline{\alpha}k}(|D^{2}P_{k,x^0}|+1)\le \overline{C}\rho^{\overline{\alpha}k}C_0 k\le \eta
\]
if $\rho$ is chosen small enough (depending only on universal constants) and $\eta$ as in Lemma \ref{lem:compactness}.
Thus $|v_{k}-u_{k}|\le\epsilon$
in $B_{1/2}\cap\{x:x_{n}>-\frac{x_n^0}{\rho^k}\}$ by Lemma \ref{lem:compactness}, where $v_{k}$ solves
\begin{equation}
\begin{cases}
F_{k}(D^{2}v_{k},x^0)=0 & \text{in }B_{1/2}\cap\{x:x_{n}>-\frac{x_n^0}{\rho^k}\},\\
v_{k}=u_{k} & \text{on }\partial (B_{1/2}\cap\{x:x_{n}>-\frac{x_n^0}{\rho^k}\}).
\end{cases}\label{eq:main-1}
\end{equation}
Since
\[
\|v_{k}\|_{L^{\infty}(B_{1/2}\cap\{x:x_{n}>-\frac{x_n^0}{\rho^k}\})}\le\|u_{k}\|_{L^{\infty}(B_{1/2}\cap\{x:x_{n}>-\frac{x_n^0}{\rho^k}\})}\le1
\]
by the maximum principle, Theorem \ref{lem:C2alpha} gives 
\begin{align}\label{eq:vkc2alpha}
\|v_{k}\|_{C^{2,\alpha}(B_{1/4}\cap\{x:x_{n}>-\frac{x_n^0}{\rho^k}\})}\le C_{0}.
\end{align}
Now define $\hat{P}_{k,x^0}$ as the second order Taylor expansion
of $v_{k}$ at the origin, and note that $F_k(D^{2}\hat{P}_{k,x^0},x^0)=F_k(D^{2}v_{k}(0),x^0)=0$.
Then
\[
|v_{k}-\hat{P}_{k,x^0}|\le C_{0}\rho^{2+\alpha}\qquad\text{in }B_{\rho}\cap\{x:x_{n}>-\frac{x_n^0}{\rho^k}\}
\]
for $\rho<1/4$, which gives
\[
|u_{k}-\hat{P}_{k,x^0}|\le|u_{k}-v_{k}|+|v_{k}-\hat{P}_{k,x^0}|\le\epsilon+C_{0}\rho^{2+\alpha}\qquad\text{in }B_{\rho}\cap\{x:x_{n}>-\frac{x_n^0}{\rho^k}\}.
\]
For $\rho^{\alpha}\le\frac{1}{2C_{0}}$ and $\epsilon\le\rho^{2}/2$,
we get
\[
|u_{k}-\hat{P}_{k,x^0}|\le\rho^{2}\qquad\text{in }B_{\rho}\cap\{x:x_{n}>-\frac{x_n^0}{\rho^k}\},
\]
or, in other words,
\[
|u-P_{k+1,x^0}|\le\rho^{2(k+1)}\qquad\text{in }B_{\rho^{k+1}}^+(x^0),
\]
for
\[
P_{k+1,x^0}(x):=P_{k,x^0}(x)+\rho^{2k}\hat{P}_{k,x^0}\left(\frac{x-x^0}{\rho^{k}}\right).
\]
Also, since $F_{k}(D^{2}\hat{P}_{k},x^0)=0$, we have
\[
F(D^{2}P_{k+1,x^0},x^0)=F(D^{2}P_{k,x^0}+D^{2}\hat{P}_{k},x^0)=F_{k}(D^{2}\hat{P}_{k},0)=0,
\]
and
\[
|D^2 P_{k+1,x^0}-D^2 P_{k,x^0}|=|D^2 \hat{P}_{k,x^0}| = |D^2 v_k(0)|\le C_0 ,
\]
by \eqref{eq:vkc2alpha}.
\end{proof}

\begin{prop}[BMO-estimate]\label{lem:BMOestimate}
Let $u$ be a viscosity solution to \eqref{eq:main}, and $P_{k,x^{0}}$ and $\rho$
be as in Lemma \ref{lem:rhoapprox}. Then
\[
\fint_{B_{\rho^{k}/2}^{+}(x^0)}|D^{2}u(y)-D^{2}P_{k,x^0}|^{2}\le C,\qquad x^0\in\overline{B}_{1/2}^{+}
\]
if $\rho$ is smaller than a constant which depends only on $\|u\|_{W^{2,p}(B_1)}$,
$f$ , $\overline{C}$ in \hyperref[hyp:Freg]{(H4)}, and universal constants.\end{prop}
\begin{proof}
Let $x^{0}\in\overline{B}_{1/2}^{+}$ and define $v(x):=u(x/R)$ and
$G(M,x):=\frac{1}{R^{2}}F(R^{2}M,\frac{x}{R})$ for $R=R(\overline{C},f, K,\delta)$
($\overline{C}$ as in \hyperref[hyp:Freg]{(H4)}) chosen so that $|G(D^{2}v,x)|\le\delta$ in $B_R^+$ 
for $\delta$ as in Lemma \ref{lem:compactness}. Note also that $\beta_{G}(x,y):=\sup_{M\in\mathcal{S}}\frac{|G(M,x)-G(M,y)|}{|M|+1}$
satisifies \hyperref[hyp:Freg]{(H4)}. Then $v$ solves
\[
\begin{cases}
G(D^{2}v,x)=\frac{f(x/R)}{R^{2}} & \text{a.e. in }B_{R}^{+}\cap(R\Omega),\\
|D^{2}v|\le\frac{K}{R^{2}} & \text{a.e. in }B_{R}^{+}\backslash(R\Omega),\\
v=0 & \text{on }B_{R}^{'},
\end{cases}
\]
and there is a polynomial $\tilde{P}_{k,x^{0}}$ for which $G(D^{2}\tilde{P}_{k,x^{0}},Rx^{0})=0$,
and a constant $\tilde{\rho}$ such that 
\[
|v(x)-\tilde{P}_{k,x^{0}}(x)|\le\tilde{\rho}^{2k},\qquad x\in B_{\tilde{\rho}^{k}}^{+}(Rx^{0}),
\]
i.e. 
\[
|u(x)-P_{k,x^{0}}(x)|\le R^2\rho^{2k},\qquad x\in B_{\rho^{k}}^{+}(x^{0}),
\]
for $P_{k,x^{0}}(x):=\tilde{P}_{k,x^{0}}(Rx)$ and $\rho^k:=\tilde{\rho}^k/R$.
Note also that
\[
F(D^{2}P_{k,x^{0}},x^{0})=F(R^{2}D^{2}\tilde{P}_{k,x^{0}},Rx^{0}/R)=R^{2}G(D^{2}\tilde{P}_{k,x^{0}},Rx^{0})=0.
\]
In particular, for $u_{k}(x):=\frac{u(\rho^{k}x+x^{0})-P_{k,x^{0}}(\rho^{k}x+x^0)}{\rho^{2k}}$,
\[
F_{k}(M,x):=F(M+D^{2}P_{k,x^{0}},\rho^{k}x+x^{0})
\]
and $\beta_{k}$
as in the proof of Lemma \ref{lem:rhoapprox}, we have $|u_{k}|\le R^2$, $\beta_{k}(x,y)\le \eta$
and $|F_{k}(u_{k},x)|\le C$. Therefore we can apply Theorem \ref{lem:W2p} to deduce
\[
\|u_{k}\|_{W^{2,p}(B_{1/2}\cap\{x_{n}\ge-x^{0}/\rho^{k}\})}\le C,
\]
or 
\[
\fint_{B_{\rho^{k}/2}^{+}(x^{0})}|D^2u(x)-D^2P_{k,x^{0}}|^{p}\, dx\le C.
\]
\end{proof}


From this it is straightforward to show that there exists a second
order polynomial $P_{r,x^0}(x)$ with $F(D^{2}P_{r,x^0},x^0)=f(x^0)$
such that 
\begin{equation*} \label{bbmo}
\sup_{r\in(0,1/4)}\fint_{B_{r}^{+}(x^0)}|D^{2}u(y)-D^{2}P_{r,x^0}|^{2}\, dy\le C,
\end{equation*}
where $x^0\in \overline{B}_{1/2}^{+}(0)$. The proof of $C^{1,1}$ regularity now follows as in \cite{IM} up to minor modifications (see also \cite{FS}). The idea is that $D^2 P_{r,x^0}(x)$ provides a suitable approximation to $D^2 u(x^0)$ and one may consider two cases: firstly, if $D^2 P_{r,x^0}(x)$ stays bounded in $r$, then one can show that $D^2 u(x^0)$ is also bounded by a constant depending only on the initial ingredients; next, if $D^2 P_{r,x^0}(x)$ blows up in $r$, one can show that the set $$A_{r}(x^0):=\frac{(B_{r}^{+}(x^0)\backslash\Omega)-x^0}{r}=B_{1}\backslash((\Omega-x^0)/r)\cap \Big\{y:y_{n}>-\frac{x_n^0}{r} \Big \}$$ decays fast enough to ensure yet again a bound on $D^2 u(x^0)$.

\bibliographystyle{amsalpha}
\bibliography{References}

\providecommand{\bysame}{\leavevmode\hbox to3em{\hrulefill}\thinspace}
\providecommand{\MR}{\relax\ifhmode\unskip\space\fi MR }
\providecommand{\MRhref}[2]{%
  \href{http://www.ams.org/mathscinet-getitem?mr=#1}{#2}
}
\providecommand{\href}[2]{#2}
\begin{thebibliography}{AMM06}

\bibitem[AG82]{MR693780}
Hans~Wilhelm Alt and Gianni Gilardi, \emph{The behavior of the free boundary
  for the dam problem}, Ann. Scuola Norm. Sup. Pisa Cl. Sci. (4) \textbf{9}
  (1982), no.~4, 571--626. \MR{693780 (85c:35089a)}

\bibitem[ALS13]{MR2999297}
John Andersson, Erik Lindgren, and Henrik Shahgholian, \emph{Optimal regularity
  for the no-sign obstacle problem}, Comm. Pure Appl. Math. \textbf{66} (2013),
  no.~2, 245--262. \MR{2999297}

\bibitem[AMM06]{MR2237208}
John Andersson, Norayr Matevosyan, and Hayk Mikayelyan, \emph{On the tangential
  touch between the free and the fixed boundaries for the two-phase
  obstacle-like problem}, Ark. Mat. \textbf{44} (2006), no.~1, 1--15.
  \MR{2237208 (2007e:35296)}

\bibitem[And07]{MR2281197}
John Andersson, \emph{On the regularity of a free boundary near contact points
  with a fixed boundary}, J. Differential Equations \textbf{232} (2007), no.~1,
  285--302. \MR{2281197 (2008e:35210)}

\bibitem[AU95]{MR1359745}
D.~E. Apushkinskaya and N.~N. Ural{\cprime}tseva, \emph{On the behavior of the
  free boundary near the boundary of the domain}, Zap. Nauchn. Sem.
  S.-Peterburg. Otdel. Mat. Inst. Steklov. (POMI) \textbf{221} (1995),
  no.~Kraev. Zadachi Mat. Fiz. i Smezh. Voprosy Teor. Funktsii. 26, 5--19, 253.
  \MR{1359745 (96m:35340)}

\bibitem[CC95]{MR1351007}
Luis~A. Caffarelli and Xavier Cabr{\'e}, \emph{Fully nonlinear elliptic
  equations}, American Mathematical Society Colloquium Publications, vol.~43,
  American Mathematical Society, Providence, RI, 1995. \MR{1351007 (96h:35046)}

\bibitem[CG80]{MR597551}
Luis~A. Caffarelli and Gianni Gilardi, \emph{Monotonicity of the free boundary
  in the two-dimensional dam problem}, Ann. Scuola Norm. Sup. Pisa Cl. Sci. (4)
  \textbf{7} (1980), no.~3, 523--537. \MR{597551 (82f:35186)}

\bibitem[FS14]{FS}
A.~{Figalli} and H.~{Shahgholian}, \emph{A general class of free boundary
  problems for fully nonlinear elliptic equations}, Arch. Ration. Mech. Anal.
  (2014), no.~1, 269--286.

\bibitem[IM]{IM}
E.~{Indrei} and A.~{Minne}, \emph{Regularity of solutions to fully nonlinear
  elliptic and parabolic free boundary problems.}, Ann. Inst. H. Poincar\'e
  Anal. Non Lin\'eaire, to appear.

\bibitem[Kry82]{MR0262675}
N.~V. Krylov, \emph{Boundedly inhomogeneous elliptic and parabolic equations},
  Izv. Akad. Nauk SSSR Ser. Mat. \textbf{46} (1982), no.~3, 487--523, 670.
  \MR{661144 (84a:35091)}

\bibitem[Mat05]{MR2180300}
Norayr Matevosyan, \emph{Tangential touch between free and fixed boundaries in
  a problem from superconductivity}, Comm. Partial Differential Equations
  \textbf{30} (2005), no.~7-9, 1205--1216. \MR{2180300 (2006f:35305)}

\bibitem[MM04]{MR2065018}
Norayr Matevosyan and Peter~A. Markowich, \emph{Behavior of the free boundary
  near contact points with the fixed boundary for nonlinear elliptic
  equations}, Monatsh. Math. \textbf{142} (2004), no.~1-2, 17--25. \MR{2065018
  (2005f:35334)}

\bibitem[PSU12]{MR2962060}
Arshak Petrosyan, Henrik Shahgholian, and Nina Uraltseva, \emph{Regularity of
  free boundaries in obstacle-type problems}, Graduate Studies in Mathematics,
  vol. 136, American Mathematical Society, Providence, RI, 2012. \MR{2962060}

\bibitem[Saf94]{Saf}
M.V. Safonov, \emph{On the boundary value problems for fully nonlinear elliptic
  equations of second order.}, Mathematics Research Report No. MRR 049-94,
  Canberra: The Australian Naitonal University (1994).

\bibitem[SU03]{MR1950478}
Henrik Shahgholian and Nina Uraltseva, \emph{Regularity properties of a free
  boundary near contact points with the fixed boundary}, Duke Math. J.
  \textbf{116} (2003), no.~1, 1--34. \MR{1950478 (2003m:35253)}

\bibitem[Wan92]{W2}
Lihe Wang, \emph{On the regularity theory of fully nonlinear parabolic
  equations. {II}}, Comm. Pure Appl. Math. \textbf{45} (1992), no.~1, 141--178.

\bibitem[Win09]{MR2486925}
Niki Winter, \emph{{$W^{2,p}$} and {$W^{1,p}$}-estimates at the boundary for
  solutions of fully nonlinear, uniformly elliptic equations}, Z. Anal. Anwend.
  \textbf{28} (2009), no.~2, 129--164. \MR{2486925 (2010i:35102)}

\end{thebibliography}



\end{document}